\documentclass[reqno,12pt]{amsart}
\usepackage{amsmath,amsthm,amsfonts,amssymb,ifpdf}

\usepackage{amssymb}
\usepackage{amsmath}
\usepackage{amsthm}
\usepackage{graphicx}
\usepackage[all]{xy}
\usepackage{enumerate}
\usepackage{tikz-cd}
\theoremstyle{plain} 
\newtheorem{theorem}{Theorem}[section]
\newtheorem{proposition}[theorem]{Proposition}
\newtheorem{lemma}[theorem]{Lemma}

\theoremstyle{definition} \newtheorem{definition}[theorem]{Definition}

\theoremstyle{remark} \newtheorem{remark}[theorem]{Remark}

\newcommand{\set}[1]{\ensuremath{\{ #1 \}}}
\newcommand{\suchthat}{\ensuremath{\, \vert \,}}

\usepackage[all]{xy}

\ifpdf
  \usepackage[
    pdftex,
    colorlinks,%
    linkcolor=blue,citecolor=red,urlcolor=blue,
    hyperindex,%
    plainpages=false,%
    bookmarksopen,%
    bookmarksnumbered%
  ]{hyperref} 
 
 \usepackage{thumbpdf}
\else
  \usepackage{hyperref}
\fi

\newcommand{\ncom}{\newcommand}
\ncom{\mylabel}[1]{{\rm (#1)}\label{#1}}
\ncom{\Hom}{{\textit{Hom}}}
\ncom{\eop}{{\hfill $\Box$}}
\begin{document}
\baselineskip=16pt

\footnotetext{Mathematics Classification Number: 14D20, 14H40, 14H52, 14F25.}

\setcounter{tocdepth}{1}

\title[Tautological classes]{Tautological algebra of the moduli space
  of semistable bundles on an elliptic curve}
  
\author[Arijit Mukherjee]{Arijit Mukherjee}
\address{School of Mathematics and Statistics\\University of Hyderabad\\Prof. C. R. Rao Road\\P.O - Central University\\Gachibowli, Hyderabad - 500046\\India.}
\email{mukherjee90.arijit@gmail.com}

\begin{abstract}
In this paper, our aim is to find the relations
amongst the cohomology classes of Brill-Noether subvarieties of the moduli
space of semistable bundles over an elliptic curve.  We obtain results similar
to the Poincar\'e relations on a Jacobian variety.
\end{abstract}
\maketitle

\tableofcontents

\section{Introduction}
\label{sec:Introduction}

Suppose $C$ is a smooth projective variety of genus $g$, defined over the complex numbers.
The Picard variety $Pic^d(C)$ is the moduli space of isomorphism classes of line bundles of degree $d$ on $C$. For $1\leq i\leq g-1$, there are naturally defined subvarieties, namely the Brill-Noether subvarieties $W_i$ of $Pic^i(C)$, which parametrize line bundles $L$ with $h^0(L)\geq 1$. Let $\Theta$ denote the theta divisor $W_{g-1}\subset Pic^{g-1}(C)$. Choosing isomorphisms $Pic^i(C)\simeq J(C)$, the classical Poincar\'e relations determine the relations between the cohomological classes of $W_i$, in $J(C)$:
$$
 [W_i]=\frac{1}{(g-i)!}[\Theta]^{g-i} \,\in\, H^*(J(C),\mathbb{Q}).
$$
See {\cite[Ch 1, \S 5,  p.~25]{ACGH}}. 
  
  Our aim in this paper is to investigate relations between the cohomology classes of the  Brill-Noether subvarieties in the moduli space of higher rank vector bundles on a curve. Brill-Noether subvarieties are more subtle in this situation and not much seems to be known in this direction. 
We consider genus one case in this paper.

Let $E$ denote a complex elliptic curve.  The moduli space of
semistable bundles over $E$ has been computed by M.~F.~Atiyah
\cite{At}.  L.~Tu \cite{T} explicitly described the Brill-Noether subvarieties
in this moduli space. In this paper, a \emph{tautological class} is
the cohomology class of Brill-Noether subvariety in the cohomology
ring and the subalgebra of the cohomology ring generated by those
tautological classes is called \emph{tautological algebra}.  We describe the relations amongst the tautological classes in the cohomology ring of the moduli space in this paper.  

We denote by $f$ the \emph{strong equivalence class} of $F$, a vector
bundle over $E$, in the sense of \cite[\S 8, p.~333]{Ses67}. Let
$\mathcal{M}_{E}(r,d)$ be the moduli space of strong equivalence
classes of semistable bundles of rank $r$ and degree $d$ over $E$.
  Let $X$
be a subvariety of $\mathcal{M}_{E}(r,d)$(or of
$\mathcal{SU}_{E}(r,L)$), then by $[X]$ we denote the cohomology class of $X$ in
$H^*(\mathcal{M}_{E}(r,d))$(respectively in $H^*(\mathcal{SU}_{E}(r,L))$).

By \cite{T}, we note that the Brill-Noether loci are
trivial for positive degree vector bundles (either empty or the whole
moduli space), and for line bundles of degree $0$ (either empty or
singleton).  Therefore, we consider the Brill-Noether loci when
degree of a vector bundle is $0$ and the rank is more than $1$.  Let
$L$ be a line bundle of degree $0$ and let $i$ be any non-negative
integer. The Brill-Noether loci inside $\mathcal{SU}_E(r,L)$ are defined as follows:
\begin{equation*}
W_{r,L}
^{i}(\exists):=\set{f\in \mathcal{SU}_E (r,L)\suchthat h^{0}(F)\geq
	i+1 \text{ for some }F\in f }.
\end{equation*}  
We denote by $[W_{r,L}^i(\exists)]$, a tautological class in
$H^*(\mathcal{SU}_{E}(r,L),\mathbb{Z})$.  We also consider the Brill-Noether
loci, denoted by $W_{r,0} ^{i}(\exists)$, inside
$\mathcal{M}_{E}(r,0)$, defined as follows:
\begin{equation*}
 W_{r,0}^{i}(\exists):=\set{f\in \mathcal{M}_E (r,0)\suchthat h^{0}(F)\geq
	i+1 \; \text{ for some }F\in f }.
\end{equation*} 
(See \cite{T}).

In \S \ref{sec:main th}, we prove the main theorems on the relations amongst the
tautological classes in
$H^*(\mathcal{SU}_{E}(r,L),\mathbb{Z})$ and in $H^*(\mathcal{M}_{E}(r,0),\mathbb{Z})$.  
We show:
\begin{theorem}
Let $r$ be any positive integer and let $L$ be a degree $0$ line
bundle over $E$. Then $W_{r,L} ^{0}(\exists)$ is a divisor inside
$\mathcal{SU}_{E}(r,L)$. Moreover, in $H^{\ast}(\mathcal{SU}_{E}(r,L),\mathbb{Z})$ we
have,
\begin{equation}\label{rel1}
[W_{r,L} ^{i}(\exists)]=[W_{r,L} ^{0}(\exists)]^{i+1}
\end{equation} 
for all $0\leq i\leq r-2$ and the tautological algebra of $\mathcal{SU}_{E}(r,L)$ is $\mathbb{Z}[\zeta]/(\zeta^r)$, where $\zeta$ is the cohomology class of $W^0_{r,L}(\exists)$ in $H^*(\mathcal{SU}_E(r,L),\mathbb{Z})$.
\end{theorem}

  Furthermore, we show that the relations amongst $[W_{r,0} ^{i}(\exists)]$'s 
 in $H^*(\mathcal{M}_{E}(r,0),\mathbb{Z})$, is similar to \eqref{rel1}.  In particular, we show:
\begin{theorem}
	The	tautological algebra of $\mathcal{M}_E(r,0)$ is
	$$
	H^*(E)\otimes \mathbb{Z}(\xi)/(\xi^r). 
	$$
	Here $\xi$ is the cohomology class of the divisor $W^0_{r,0}(\exists)$ on $\mathcal{M}_E(r,0)$ in $H^*(\mathcal{M}_E(r,0),\mathbb{Z})$. 
\end{theorem} 
 
 It is interesting to investigate if similar relations hold in higher genus case.

%%%%%%%%%%%%%%%%%%%%%%%%%%%%%%%%%%%%%%%%%%%
\section{Notations}

All the varieties are defined over the complex numbers.

1)  $\mathcal{SU}_{E}(r,L)$ denotes the moduli space of strong equivalence
classes of semistable bundles of rank $r$ and fixed determinant $L$ of degree $d$ over $E$.

2)  $S^nE$ is
the $n$th symmetric product of the curve $E$.  

3)
 $J(E)$ is the Jacobian variety of isomorphism classes of line bundles of
degree $0$ over $E$.  

4) $J_{d}(E)$ is the
isomorphism classes of line bundles of degree $d$ over $E$.

5) $det F$ is the determinant bundle of a semistable vector bundle $F$ over $E$.

6) $\mathcal{O}(D)$ denotes the line bundle corresponding to a divisor $D$ on $E$.

7) $\mathcal{O}_E$ is the trivial line bundle over $E$.

8) $h^i(E,F)$ (or $h^i(F)$) denotes the dimension of $H^i(E,F)$ for a semistable bundle $F$ over $E$.

9) By $H^{\ast}(X)$ we mean $H^{\ast}(X,\mathbb{Z})$, the cohomology ring of a complex variety $X$ with integral coefficients. 

%%%%%%%%%%%%%%%%%%%%%%%%%%%%%%%%%%%%%%%%%%%%

\section{Brill-Noether loci of semistable bundles over an elliptic curve}
\label{sec:BN-loci}

\subsection{Bundles with positive degree}

In this section we recall a few results from \cite{T}
which will be relevant for the next section.  The following
lemma is proved in \cite{T}.  For the sake of completeness we include the lemma with proof here.

\begin{lemma}\label{L1}{\cite[Lemma 17, p.~13]{T}}
Any semistable bundle $F$ of positive degree over $E$ is non-special,
that is, $h^1(F)=0$.
\end{lemma}
\begin{proof}
Let $F$ be a semistable vector bundle of degree $d>0$ and $K_E$ be the
canonical line bundle over $E$.  Then $K_E\cong \mathcal{O}_E$ as $E$
is an elliptic curve.  By Serre duality we have
\begin{equation*}
h^1(F)=h^0(K_E\otimes F^{\ast})=h^0(F^{\ast})
\end{equation*}
As $F^{\ast}$ is also a semistable bundle and of negative degree,
$h^0(F^{\ast})=0$.  Therefore, $h^1(F)=0$.  Moreover by Riemann-Roch
theorem, $h^0(F)=d$.
\end{proof}
\begin{remark}\label{R1}
A consequence of Lemma \ref{L1} is that the map 
\begin{equation*}
\begin{split}
h^{0}:\mathcal{M}_E (r,d)&\longrightarrow \mathbb{Z}_{+}\cup
        \{0\} \\ 
f&\mapsto h^0(F)
\end{split}
\end{equation*}
is well defined for $d>0$, and is the constant function $d$.
\end{remark}
\begin{definition}
Let $d>0$ and $i\geq 0$ be any two integer.  The \emph{Brill-Noether
  loci} are defined by
\begin{equation*}
W^{i}_{r,d}:=\set{f\in \mathcal{M}_{E}(r,d)\suchthat h^{0}(F)\geq i+1 }.
\end{equation*}
This definition is well defined by Remark \ref{R1}.
\end{definition}
The following lemma is a direct consequence of Lemma \ref{L1}.  See
\cite[p.~13]{T}.
\begin{lemma}\label{L2}
Let $d>0$.  Then
\begin{equation*} 
	W^{i}_{r,d} \cong \left\{ \begin{array}{ll}
	\emptyset & \text{if } 1\leq d \leq i;\\
	\mathcal{M}_{E}(r,d) & \text{if } d\geq i+1.\end{array} \right. 
\end{equation*}
\end{lemma}

Therefore Brill-Noether loci inside $\mathcal{M}_{E}(r,d)$ are not of
much interest when $d>0$.  Moreover for degree $0$ line bundles over
$E$ we have the following result.

\begin{lemma}
\label{L3}  The Brill-Noether loci for $d=0$, $r=1$ are 
\begin{equation*} 
W^{i}_{1,0} \cong \left\{ \begin{array}{ll}
	\emptyset & \mbox{if $1\leq i$};\\ \{\mathcal{O}_E\} &
        \mbox{if $i=0$}.\end{array} \right.
\end{equation*}
\end{lemma}
\begin{proof}
See \cite[p.~13]{T}.  As $h^0(L)=0$ or $1$ for a line
bundle $L$ of degree zero over $E$ and moreover $h^0(L)=1$ if and only
if $L\cong \mathcal{O}_E$.
\end{proof}

\subsection{Degree zero bundles}

When $d=0$, $h^{0}:\mathcal{M}_E
(r,0)\longrightarrow \mathbb{Z}_{+}\cup \{0\} $ is not well defined.
For example, let $F_2$ be the \emph{Atiyah's indecomposable bundle} of
rank $2$ and $I_{2}$ be the trivial
bundle of rank $2$.  Then $F_{2}\cong I_{2}$, but $h^{0}(F_{2})=1\neq
2= h^{0}(I_{2})$, \cite[proof of Theorem 5, p.~432]{At}. 

In this case,  
Brill-Noether loci inside $\mathcal{SU}_E(r,L)$ and $\mathcal{M}_E (r,0)$ are defined, see \cite[p.~5]{T}. Here $L$ is a line bundle on $E$ of degree $0$.
 
\begin{definition}  
\begin{equation*}
\begin{split}
W_{r,L}
        ^{i}(\exists)& :=\set{f\in \mathcal{SU}_E(r,L)\suchthat h^{0}(F)\geq
          i+1 \text{ for some }F\in f }.\\
	 W_{r,0}
        ^{i}(\exists)& :=\set{f\in \mathcal{M}_E (r,0)\suchthat h^{0}(F)\geq
        i+1 \; \text{ for some }F\in f }
	\end{split}
	\end{equation*}
\end{definition}

We have the equality:
\begin{equation*}
W_{r,L} ^{i}(\exists)= W_{r,0} ^{i}(\exists) \cap \mathcal{SU}_E (r,L).
\end{equation*}

We now have the following isomorphism of moduli spaces.
\begin{proposition}\label{P1}
	Let $r>1$ be a positive integer and $L$ a degree $0$ line bundle over $E$.  Then we have the isomorphisms:
	\begin{equation*}
	\begin{split}
	\mathcal{M}_E(r,0)&\cong S^r E\\
	\mathcal{SU}_{E}(r,L)&\cong \mathbb{P}^{r-1}\\
	J(E)&\cong E\\
	W_{r,0}^i(\exists)&\cong S^{r-i-1} E\\
	W_{r,L}^i(\exists)&\cong \mathbb{P}^{r-i-2}.
	\end{split}
	\end{equation*}
\end{proposition}
\begin{proof} See
{\cite[Theorem 2, 3, 4 and 5; Corollary 15; p.~4-5, 10]{T}}.

\end{proof}

%%%%%%%%%%%%%%%%%%%%%%%%%%%%%%%%%%%%%%%%%%%%%%%%%
\section{Tautological algebra of semistable bundles of degree zero over an elliptic curve}
\label{sec:tautological-alg}

Consider the map:
\begin{equation}\label{E1}
\begin{split}
\pi : J(E)\times \mathcal{SU}_{E}(r,L) &\rightarrow \mathcal{M}_E(r,d)\\ (l,F)
&\mapsto l\otimes F.
\end{split}
\end{equation}
See \cite[p.~338]{BT} and \cite[p.~348]{DT}.

This map as in \eqref{E1} suggests that the cohomology subalgebra generated by $[W^i_{r,0}(\exists)]$'s is the same as that generated by $[W^i_{r,L}(\exists)]$'s with coefficients lying in $H^*(J(E))$. We make this precise in the next section. \\

%We will utilise the following commutative diagrams in the proofs, in next section.\\
Let $det: \mathcal{M}_E(r,0) \rightarrow J(E)$ be the \emph{determinant} map which sends $F$ to $det F$ and $\alpha:S^{r}E \rightarrow J_{r}(E)$ be the \emph{Abel-Jacobi} map defined as $x_1+x_2+\cdots +x_r \mapsto \mathcal{O}(x_1+x_2+\cdots +x_r)$.  Then we have the following commutative diagram (\cite[p.~12]{T}), which says that the \emph{determinant} map and the \emph{Abel-Jacobi} map can be identified.
\begin{equation}\label{diagram 1}
\begin{tikzcd}
\mathcal{M}_E(r,0) \arrow{r}{\cong} \arrow{d}{det}
&S^{r}E \arrow{d}{\alpha}\\
J(E) \arrow{r}{\cong} &J_{r}(E)
\end{tikzcd}
\end{equation}
\begin{remark}\label{R2}
By \emph{Abel}'s theorem (\cite[p.~18]{ACGH}), fiber of the map $\alpha$ over any line bundle $L\in J_r(E)$ is the complete linear system $\mid D\mid$ of a divisor $D$ on $E$ with $\mathcal{O}(D)=L$.  Now if $r>0$, then by Serre duality $h^1(E,\mathcal{O}(D))=0$ and by Riemann-Roch theorem $h^0(E,\mathcal{O}(D))=r$.  Therefore each fiber of the map $\alpha:S^{r}E \rightarrow J_{r}(E)$ is isomorphic to $\mathbb{P}^{r-1}$ if $r>0$.  This is also another reason to work in $\mathcal{M}_E(r,0)$ with $r>0$.
\end{remark}
Let $det: W^i_{r,0}(\exists)\rightarrow J(E)$ be the restriction of the \emph{determinant} map $det: \mathcal{M}_E(r,0) \rightarrow J(E)$.  Then we have the following commutative diagram (\cite[p.~16]{T}) similar to (\ref{diagram 1}).

\begin{equation}\label{diagram 2}
\begin{tikzcd}
W^i_{r,0}(\exists) \arrow{r}{\cong} \arrow{d}{det}
&S^{r-i-1}E \arrow{d}{\alpha}\\
J(E) \arrow{r}{\cong} &J_{r-i-1}(E)
\end{tikzcd}
\end{equation}
Here we assume $0\leq i\leq r-2$.  So by Remark \ref{R2}, the fiber of $\alpha$ is isomorphic to $\mathbb{P}^{r-i-2}$.

%%%%%%%%%%%%%%%%%%%%%%%%%%%%%%%%%%%%%%%%%%%%%%%%%%%%%%%%%%%%%%%%%%%%%%%

\section{Main theorems}
\label{sec:main th}
In this section we define the tautological subalgebra of $H^*(\mathcal{M}_{E}(r,0),\mathbb{Z})$ and $H^*(\mathcal{SU}_{E}(r,L),\mathbb{Z})$ and prove our main theorems.
\begin{definition}
	The cohomology classes $[W^{i}_{r,0}(\exists)]\in H^*(\mathcal{M}_{E}(r,0),\mathbb{Z})$ are called the \emph{tautological classes}.  The subalgebra of $H^*(\mathcal{M}_{E}(r,0),\mathbb{Z})$ generated by these tautological classes is called the \emph{tautological subalgebra} of $H^*(\mathcal{M}_{E}(r,0),\mathbb{Z})$.
\end{definition}

\begin{definition}
	Let $L$ be a degree zero line bundle over $E$.  Then the cohomology classes $[W^{i}_{r,L}(\exists)]\in
	H^*(\mathcal{SU}_{E}(r,L),\mathbb{Z})$ are called the \emph{tautological classes}.  The subalgebra of  \\ $H^*(\mathcal{SU}_{E}(r,L),\mathbb{Z})$ generated by these tautological classes is called the \emph{tautological subalgebra} of $H^*(\mathcal{SU}_{E}(r,L),\mathbb{Z})$.
\end{definition}

Following theorem shows that the tautological class $\zeta:=[W_{r,L} ^{0}(\exists)]$ is the generator of the tautological subalgebra of $H^*(\mathcal{SU}_{E}(r,L),\mathbb{Z})$.
\begin{theorem}\label{T1}
	Let $r$ be any positive integer and let $L$ be a degree $0$ line bundle over $E$. Then $W_{r,L} ^{0}(\exists)$ 
	is a divisor inside $\mathcal{SU}_{E}(r,L)$. Moreover, in $H^{\ast}(\mathcal{SU}_{E}(r,L),\mathbb{Z})$ we have, 
	\begin{equation*}
	[W_{r,L} ^{i}(\exists)]=[W_{r,L} ^{0}(\exists)]^{i+1}
	\end{equation*} 
	for all $0\leq i\leq r-2$ and the tautological algebra of $\mathcal{SU}_{E}(r,L)$ is $\mathbb{Z}[\zeta]/(\zeta^r)$, where $\zeta$ is the cohomology class of $W^0_{r,L}(\exists)$ in $H^*(\mathcal{SU}_E(r,L),\mathbb{Z})$.
	
\end{theorem}

\begin{proof}
	We have the following stratification inside $\mathcal{SU}_{E}(r,L)$ by Proposition \ref{P1}.
	\[
	\begin{matrix}
	&& \mathcal{SU}_{E}(r,L) & \cong & \mathbb{P}^{r-1} \\
	&& \rotatebox{90}{$\subseteq$} && \rotatebox{90}{$\subseteq$} \\
	&&  W_{r,L} ^{0}(\exists) & \cong &\mathbb{P}^{r-2} \\
	&& \rotatebox{90}{$\subseteq$} && \rotatebox{90}{$\subseteq$} \\
	&&  W_{r,L} ^{1}(\exists) & \cong & \mathbb{P}^{r-3}\\
	&& \rotatebox{90}{$\subseteq$} && \rotatebox{90}{$\subseteq$} \\
	&& \cdot && \cdot \\
	&& \cdot && \cdot \\
	&& \cdot && \cdot \\
	&& \rotatebox{90}{$\subseteq$} && \rotatebox{90}{$\subseteq$} \\
	&&  W_{r,L} ^{r-3}(\exists) & \cong & \mathbb{P}^{1}\\
	&& \rotatebox{90}{$\subseteq$} && \rotatebox{90}{$\subseteq$} \\
	&&  W_{r,L} ^{r-2}(\exists) & \cong & \mathbb{P}^{0} & \cong & \{\cdot \}
	\end{matrix}
	\]
So, $W_{r,L} ^{0}(\exists)$ is a subvariety of $\mathcal{SU}_{E}(r,L)$ of
codimension $1$ and hence a divisor.  We can calculate relations
between $[\mathbb{P}^{i}]$'s as follows. Inside $\mathbb{P}^{r-1}$ we
have the following stratification:
	\begin{equation*}
	\{\cdot \}\subseteq \mathbb{P}^{1}\subseteq \mathbb{P}^{2}\subseteq \cdots \subseteq \mathbb{P}^{r-2}\subseteq \mathbb{P}^{r-1}.
	\end{equation*}
	Then we have:
	\begin{equation}\label{E2}
	H^{\ast}(\mathbb{P}^{r-1},\mathbb{Z})=\dfrac{\mathbb{Z}[\zeta]}{<\zeta^{r}>}
	\end{equation}
	 where $\zeta$ is the cohomology class of $\mathbb{P}^{r-2}$, that is, $\zeta=[\mathbb{P}^{r-2}]=c_1(\mathcal{O}(1))$, $c_1(\mathcal{O}(1))$ being the first chern class of $\mathcal{O}(1)$ over $\mathbb{P}^{r-1}$.  Moreover in $H^{\ast}(\mathbb{P}^{r-1},\mathbb{Z})$, we have:
	 \begin{equation}\label{E3}
	 [\mathbb{P}^{r-1-k}]=\zeta^{k}.
	 \end{equation}
	 Therefore, by \eqref{E2} and Proposition \ref{P1} we get:
	 \begin{equation*}
	 H^*(\mathcal{SU}_E(r,L),\mathbb{Z})=\dfrac{\mathbb{Z}[\zeta]}{<\zeta^{r}>}.
	 \end{equation*}
	Furthermore, by \eqref{E3} and Proposition \ref{P1} we get the following equality in $H^*(\mathcal{SU}_E(r,L),\mathbb{Z})$:
	\begin{equation*}
	[W_{r,L} ^{i}(\exists)]= [\mathbb{P}^{r-i-2}]=[\mathbb{P}^{r-1-(i+1)}]=\zeta^{i+1}=[\mathbb{P}^{r-2}]^{i+1}=[W_{r,L} ^{0}(\exists)]^{i+1}. 
	\end{equation*}
	Hence the theorem follows.
\end{proof}

The next theorem is about some relations between the generators of the tautological subalgebra of $H^{\ast}(\mathcal{SU}_{E}(r,L),\mathbb{Z})$ and $H^{\ast}(\mathcal{M}_{E}(r,0),\mathbb{Z})$.  The theorem says that tautological algebra of $\mathcal{M}_E(r,0)$ is generated by the cohomology class of the Brill-Noether subvariety $W^0_{r,0}(\exists)$ as an $H^*(E)$-algebra.

\begin{theorem}\label{T2}
The	tautological algebra of $\mathcal{M}_E(r,0)$ is
$$
H^*(E)\otimes \mathbb{Z}(\xi)/(\xi^r). 
$$
Here $\xi$ is the cohomology class of the divisor $W^0_{r,0}(\exists)$ on $\mathcal{M}_E(r,0)$ in $H^*(\mathcal{M}_E(r,0),\mathbb{Z})$. 
\end{theorem}
\begin{proof}
We have the following stratification inside $\mathcal{M}_{E}(r,0)$  by Proposition \ref{P1}.
\[
\begin{matrix}
&& \mathcal{M}_{E}(r,0) &  \cong & S^{r}E \\
&& \rotatebox{90}{$\subseteq$} && \rotatebox{90}{$\subseteq$} \\
&&  W_{r,0} ^{0}(\exists) & \cong & S^{r-1}E \\
&& \rotatebox{90}{$\subseteq$} && \rotatebox{90}{$\subseteq$} \\
&&  W_{r,0} ^{1}(\exists) & \cong & S^{r-2}E\\
&& \rotatebox{90}{$\subseteq$} && \rotatebox{90}{$\subseteq$} \\
&& \cdot && \cdot \\
&& \cdot && \cdot \\
&& \cdot && \cdot \\
&& \rotatebox{90}{$\subseteq$} && \rotatebox{90}{$\subseteq$} \\
&&  W_{r,0} ^{r-2}(\exists) & \cong & S^{1}E & \cong & E

\end{matrix}
\] 
So, $W_{r,0} ^{0}(\exists)$ is a subvariety of
$\mathcal{M}_{E}(r,0)$ of codimension $1$ and hence a divisor.  By \eqref{diagram 1} and by Remark \ref{R2}, the determinant morphism 
$$
\mathcal{M}_E(r,0)\rightarrow J(E)
$$
is a projective bundle $\mathbb{P}^{r-1}_{J(E)} \rightarrow J(E)$.

Hence, by projective bundle formula,
\begin{equation}\label{E4}
H^*(\mathcal{M}_E(r,0),\mathbb{Z})\,=\,H^*(J(E))\otimes \mathbb{Z}(\xi)/(\xi^r).
\end{equation}
Here $\xi$ is the first Chern class of $\mathcal{O}(1)$ on $\mathbb{P}^{r-1}_{J(E)}$.

Therefore by \eqref{E4} and Proposition \ref{P1} we get,
\begin{equation*}
H^*(\mathcal{M}_E(r,0),\mathbb{Z})\,=\,H^*(E)\otimes \mathbb{Z}(\xi)/(\xi^r).
\end{equation*}
However, by \eqref{diagram 2}, we have the equality of the cohomology classes:
$$
[W^i_{r,0}(\exists)]=\xi^{i+1}
$$
for all $0\leq i\leq r-2$.
This gives the assertion.

\end{proof}

\vspace*{.1 in}

\section*{Acknowledgement}
\label{sec:Acknowledgement}
I would like to express my sincere gratitude to Prof. Jaya N. N. Iyer for introducing this problem to me and for her valuable suggestions.  I would like to thank Dr. Archana S Morye and Dr. Tathagata Sengupta for their continuous guidance and encouragements throughout the work.  I would also like to thank University Grants Commission (UGC) (ID~- ~424860) for financial support.

\end{document}